\documentclass[preprint]{elsarticle}

\usepackage{amssymb}
\usepackage{amsthm}
\usepackage{amsfonts}
\usepackage{amsmath}
\usepackage{color}
\usepackage{natbib}
\usepackage{bbm}            

\newcommand{\Petkovsek} {{Petkov\v{s}ek}}

\newcommand{\bN} { {\mathbb{N}}}

\newcommand{\bZ} { {\mathbb{Z}}}

\newcommand{\bF} { {\mathbb{F}}}
\newcommand{\bE} { {\mathbb{E}}}
\newcommand{\bK} { {\mathbb{K}}}
\newcommand{\bO} { {\mathbb{O}}}
\newcommand{\bW} { {\mathbb{W}}}
\newcommand{\bU} { {\mathbb{U}}}
\newcommand{\bS} { {\mathbb{S}}}

\def\bm#1{\mathchoice{\kern-.5pt\mathord{\text{\bfseries\itshape#1}}\kern+.25pt}
                      {\kern-.5pt\mathord{\text{\bfseries\itshape#1}}\kern+.25pt}
                      {\kern+.25pt\mathord{\text{\scriptsize\bfseries\itshape#1}}\kern+.5pt}
                      {\kern+.25pt\mathord{\text{\tiny\bfseries\itshape#1}}}\kern+.25pt}

\newcommand{\vk} { {\bm k}}
\newcommand{\vt} { {\bm t}}

\newcommand{\vy} { {\bm y}}
\newcommand{\vz} { {\bm z}}
\newcommand{\vu} { {\bm u}}
\newcommand{\vv} { {\bm v}}
\newcommand{\vw} { {\bm w}}
\newcommand{\vnull} { {\bm 0}}
\newcommand{\vD}{{\bm D}}
\newcommand{\vS}{{\bm S}}

\newcommand{\si} { {\sigma}}

\newcommand{\pa}{\partial}

\def\Wt{\bW_{\!\vt}}
\def\Wtz{\bW_{\!\vt,\vz}}
\def\Ot{\bO_{\vt}}
\def\Otz{\bO_{\vt,\vz}}
\def\Dt{\vD_{\!\vt}}
\def\Dy{\vD_{\!\vy}}
\def\Dz{\vD_{\!\vz}}
\def\Sk{\vS_{\!\vk}}
\def\eqa{\stackrel{\mathrm{alg}}{=}}
\def\neqa{\,\not\!\stackrel{\mathrm{alg}}{=}}

\newtheorem{theorem}{Theorem}
\newtheorem{prop}[theorem]{Proposition}

\newtheorem{remark}[theorem]{Remark}
\newtheorem{define}[theorem]{Definition}

\newtheorem{conj}[theorem]{Conjecture}

\journal{xxxxxx}

\begin{document}

\begin{frontmatter}



\title{Proof of the Wilf--Zeilberger Conjecture for \\Mixed Hypergeometric Terms\footnote{S.\ Chen was supported by the NSFC grants 11501552, 11688101 and
by the Frontier Key Project (QYZDJ-SSW-SYS022) and the Fund of the Youth Innovation Promotion Association, CAS. C.\ Koutschan was supported by the Austrian Science Fund (FWF): W1214 and P29467-N32.
This work was also supported by the Fields Institute's 2015 Thematic Program on Computer Algebra in Toronto, Canada.}
}

\author[a,b]{Shaoshi Chen}
\address[a]{%
  KLMM, Academy of Mathematics and Systems Science, \\
  Chinese Academy of Sciences, Beijing 100190, China}
\address[b]{%
  School of Mathematical Sciences, \\
  University of Chinese Academy of Sciences, Beijing 100049, China}
\ead{schen@amss.ac.cn}

\author[c]{Christoph Koutschan}
\address[c]{Johann Radon Institute for Computational and Applied Mathematics (RICAM),\\
  Austrian Academy of Sciences, Altenberger Stra\ss e 69, 4040 Linz, Austria}
\ead{christoph.koutschan@ricam.oeaw.ac.at}

\address{{\bf Dedicated to Professor Peter Paule on the occasion of his 60th birthday}}

\begin{abstract}
In 1992, Wilf and Zeilberger conjectured that a hypergeometric term in several
discrete and continuous variables is holonomic if and only if it is proper.
Strictly speaking the conjecture does not hold, but it is true when
reformulated properly: Payne proved a piecewise interpretation in 1997, and
independently, Abramov and Petkov\v{s}ek in 2002 proved a conjugate
interpretation. Both results address the pure discrete case of the conjecture.
In this paper we extend their work to hypergeometric terms in several discrete
and continuous variables and prove the conjugate interpretation of the
Wilf--Zeilberger conjecture in this mixed setting.
\end{abstract}

\begin{keyword}
Wilf--Zeilberger Conjecture \sep
hypergeometric term \sep
properness \sep
holonomic function \sep
D-finite function \sep
Ore-Sato Theorem
\end{keyword}

\end{frontmatter}


\section{Introduction}\label{SECT:intro}

The method of creative telescoping was put on an algorithmic fundament by
Zeilberger~\cite{Zeilberger1990, Zeilberger1991} in the early 1990's, and it
has been a powerful tool in the study of special function identities since
then. Zeilberger's algorithms (for binomial / hypergeometric sums and for
hyperexponential integrals) terminate for \emph{holonomic} inputs. The
holonomicity of functions is defined in terms of the dimension of their
annihilating ideals in algebras of operators with polynomial coefficients;
in general, it is difficult to detect this property. In
1992, Wilf and Zeilberger~\cite{Wilf1992} gave a more elegant and constructive
proof that their methods are applicable to so-called \emph{proper
  hypergeometric} terms, which are expressed in an explicit form. Since all
examples considered in their paper are both proper and holonomic, Wilf and
Zeilberger then presented the following conjecture in~\cite[p.~585]{Wilf1992}.
\begin{conj}[Wilf and Zeilberger, 1992]\label{CONJ:wz}
A hypergeometric term is holonomic if and only if it is proper.
\end{conj}
It was observed in \cite{Payne1997,AbramovPetkovsek2002a} that the conjecture
is not true when it is taken literally, so it needs to be modified in order to
be correct (see Theorem~\ref{THM:main}).  For example, the term $|k_1-k_2|$
defining a sequence over~$\bN^2$ is easily seen to be hypergeometric (since it
satisfies the first-order system below) and holonomic because its generating
function
\[
  \sum_{k_1=0}^\infty \sum_{k_2=0}^\infty |k_1-k_2| z_1^{k_1} z_2^{k_2} =
  \frac{z_1^2z_2 + z_1z_2^2 - 4z_1z_2 + z_1 + z_2}{(1-z_1)^2(1-z_2)^2(1-z_1z_2)}
\]
is a rational function (see Definitions~\ref{DEF:fgf} and~\ref{DEF:holonomic},
and Theorem~\ref{THM:kashi}). But $|k_1-k_2|$ is not
proper~\cite[p.~396]{AbramovPetkovsek2002a}; a similar counter-example was
given in~\cite[p.~55]{Payne1997}.  Payne in his 1997
Ph.D. dissertation~\cite[Chap.~4]{Payne1997} modified and proved the
conjecture in a piecewise sense; more specifically, it was shown that the
domain of a holonomic hypergeometric term can be expressed as the union of a
linear algebraic set and a finite number of convex polyhedral regions (the
``pieces'') such that the term is proper on each region. In the case of $|k_1
- k_2|$, the linear algebraic set is the line $k_1=k_2$ and the polyhedral
regions are $k_1 - k_2 > 0$ and $k_1 - k_2 < 0$ where the proper terms are
$k_1 - k_2$ and $k_2 - k_1$ respectively.  Unaware of~\cite{Payne1997},
Abramov and \Petkovsek~\cite{AbramovPetkovsek2002a} in 2002 solved the problem
by showing that a holonomic hypergeometric term is conjugate to a proper term,
which means roughly that both terms are solutions to a common (nontrivial)
system of equations.  The holonomic term $|k_1 - k_2|$ and the proper term
$k_1 - k_2$ are easily seen to be solutions of the system
\begin{align*}
  (k_1 - k_2)T(k_1 + 1, k_2) - (k_1 - k_2 + 1)T(k_1, k_2) &=0, \\
  (k_1 - k_2)T(k_1, k_2 + 1) - (k_1 - k_2 - 1)T(k_1, k_2) & =0.
\end{align*}

The special case of two variables has also been shown by
Hou~\cite{Hou2001, Hou2004} and by Abramov and Petkov\v{s}ek~\cite{AbramovPetkovsek2001}.
In this paper, we consider the general mixed case, but only the conjugate interpretation.
For the sake of simplicity, we regard hypergeometric terms as literal functions
only of the discrete variables and interpret their values as elements of a differential field.
We define exponentiation of these field elements only in a formal
sense, similar to what is done in symbolic integration, see Remark~\ref{remark:exponentiation}.

If Conjecture~\ref{CONJ:wz} above were verified, then one could algorithmically detect the
holonomicity of hypergeometric terms by checking properness with the
algorithms in~\cite{AbramovPetkovsek2002a, Chen2011}.
This is important because it gives a simple test for
the termination of Zeilberger's algorithm. In the bivariate case,
several termination criteria are developed
in~\cite{Abramov2003, Chen2005, Chen2015}.

\section{Mixed hypergeometric terms}\label{SECT:hf}
Hypergeometric terms play a prominent role in combinatorics; and also a
large class of special functions used in mathematics and physics can be defined
in terms of them, namely as hypergeometric series.  Wilf and Zeilberger
in~\cite{WilfZeilberger1990a, WilfZeilberger1990b, Wilf1992} developed an
algorithmic proof theory for identities involving hypergeometric terms.

Throughout the paper, we let $\bF$ denote an algebraically closed field of characteristic zero, and
$\vt =(t_1,\ldots,t_m)$ and $\vk=(k_1,\ldots, k_n)$ be two sets of variables;
we view~$\vt$ and $\vk$ as continuous and discrete variables,
respectively.  Note that bold symbols are used for vectors and that
$\vu\cdot\vv=u_1v_1+\dots+u_nv_n$ denotes their inner product.  For an
element~$f\in \bF(\vt, \vk)$, define
\[
  D_i(f)=\frac{\pa f}{\pa t_i} \quad \text{and} \quad
  S_j(f(\vt, \vk)) = f(\vt, k_1, \ldots, k_{j-1}, \, k_j+1, \, k_{j+1}, \ldots, k_n)
\]
for all~$i, j$ with~$1 \le i \le m$ and~$1 \le j \le n$. The
operators~$D_i$ and~$S_j$ are called \emph{derivations} and
\emph{shifts}, respectively.  The operators
$D_1, \ldots, D_m, S_1, \ldots, S_n$ commute pairwise on~$\bF(\vt, \vk)$.

The field~$\bF(\vt)$ becomes a differential field~\cite[p.~58]{Kolchin1973}
with the derivations $D_1, \ldots, D_m$.  Let~$\bU$ be a
universal differential extension of~$\bF(\vt)$, in which all consistent systems of algebraic
differential equations with coefficients in~$\bF(\vt)$ have solutions and the
extended derivations~$D_1, \ldots, D_m$ commute in~$\bU$. For the existence of
such a universal field, see Kolchin's book~\cite[p.~134, Theorem
2]{Kolchin1973}. By a multivariate sequence over~$\bU$, we mean a map~$H\colon
\bN^n \rightarrow \bU$; instead of $H(\vk)$ we will often write
$H(\vt,\vk)$ in order to emphasize also the dependence on~$\vt$. Let~$\bS$
be the set of all multivariate sequences over~$\bU$.
We define the addition and multiplication of two elements of~$\bS$
coordinatewise, so that the invertible elements in~$\bS$ are those sequences whose
entries are all invertible in~$\bU$. The shifts $S_j$ operate on sequences in an
obvious way, and the derivations on~$\bU$ are extended to~$\bS$ coordinatewise.

A polynomial $p\in\bF[\vt, \vk]$ can be viewed as an element of~$\bS$ in a
natural way: for each $\vv\in\bN^n$ the entry of the sequence is given by the
value $p(\vt,\vv)$. However, in order to embed the field of rational
functions~$\bF(\vt, \vk)$ into~$\bS$, we need the following equivalence
relation among multivariate sequences, used in~\cite{AbramovPetkovsek2002a}.
\begin{define}[Equality modulo an algebraic set] \label{DEF:equivalence}
Two multivariate sequences $H_1(\vt,\vk)$ and~$H_2(\vt,\vk)$ are said to be
\emph{equal modulo an algebraic set}, denoted by~$H_1 \eqa H_2$,
if there is a nonzero polynomial~$p\in \bF[k_1, \dots, k_n]$ such that
\[
  \{{\vv}\in \bN^n\mid H_1(\vt,\vv) \neq H_2(\vt,\vv)\} \subseteq \{\vv\in \bF^n \mid p(\vv)=0\}.
\]
A multivariate sequence~$H$ is \emph{nontrivial} if~$H\neqa 0$.
\end{define}

Equality modulo an algebraic set is not only an equivalence relation in~$\bS$,
but also a congruence~\cite{AbramovPetkovsek2002a}, i.e., $H_1 + H_2 \eqa H_1'
+ H_2'$ and~$H_1H_2 \eqa H_1'H_2'$ if~$H_1 \eqa H_1'$ and~$H_2 \eqa H_2'$.
Now, to every rational function we can associate a sequence in~$\bS$.
\begin{define}\label{DEF:RationalSequence}
Let $f=p/q$ with $p,q\in \bF[\vt,\vk]$ be a rational function. Then we
define the \emph{rational sequence}~$F(\vt,\vk)\in\bS$
corresponding to~$f(\vt,\vk)$ as follows: For each $\vv\in\bN^n$
\[
  F(\vt,\vv) = \begin{cases}
    f(\vt,\vv), & \text{if } q(\vt,\vv) \neq 0, \\
    0, & \text{otherwise.}
  \end{cases}
\]
\end{define}

\begin{define}[Mixed hypergeometric term] \label{DEF:hf}
A multivariate sequence~$H(\vt,\vk)\in \bS$ is said to be a \emph{(mixed) hypergeometric term} over~$\bF(\vt, \vk)$ if
there are polynomials~$p_i, q_i\in \bF[\vt, \vk]$ with~$q_i\neq 0$ for $i=1, \ldots, m$ and
$u_j, v_j\in \bF[\vt, \vk]\setminus \{0\}$ for $j=1, \ldots, n$
such that
\[
  q_i D_i(H) = p_i H \quad\text{and}\quad v_j S_j(H) = u_j H.
\]
\end{define}
Strictly speaking, such sequences should be called \emph{hypergeometric-hyper\-expo\-nential},
but for the sake of brevity we call them just \emph{hypergeometric}.
Let~$a_i(\vt,\vk)= p_i/q_i$ and~$b_j(\vt,\vk) = u_j/v_j$ with~$p_i, q_i, u_j, v_j$
as in the above definition. Then we can write
\[
  D_i(H) \eqa a_i H \quad\text{and}\quad S_j(H) \eqa b_j H.
\]
We call the rational functions~$a_i$ and~$b_j$ the~\emph{certificates}
of~$H$. The certificates of a hypergeometric term are not arbitrary
rational functions. They satisfy certain compatibility conditions. The
following definition is a continuous-discrete extension of the one introduced
in~\cite{AbramovPetkovsek2002a}.

\begin{define}[Compatible rational functions] \label{DEF:crf}
We call the rational functions $a_1, \ldots, a_m\in \bF(\vt, \vk)$, $b_1, \ldots, b_n \in \bF(\vt, \vk)\setminus\{0\}$
\emph{compatible} with respect to~$\{D_1, \ldots, D_m, S_1, \ldots, S_n\}$
if the following three groups of conditions hold:
\begin{alignat}{3}
D_i(a_j) & =D_j(a_i), & \quad & 1 \le i < j \le m, \label{EQ:dd} \\
\frac{S_i(b_j)}{b_j} & = \frac{S_j(b_i)}{b_i}, &&  1 \le i < j \le n, \label{EQ:ss} \\
\frac{D_i(b_j)}{b_j} & = S_j(a_i) - a_i, && 1 \le i \le m \text{ and } 1 \le j \le n.\label{EQ:ds}
\end{alignat}
\end{define}
\begin{remark}
Let~$H\in \bS$ be a nontrivial hypergeometric term over~$\bF(\vt,\vk)$.
By the same argument as in the proof of~\cite[Prop.~4]{AbramovPetkovsek2002a},
we have that the certificates of~$H$ are unique (if we take the reduced
form of rational functions) and compatible with respect to~$\{D_1, \ldots, D_m, S_1, \ldots, S_n\}$.
\end{remark}

\section{Structure of compatible rational functions}\label{SEC:scrf}
The structure of rational solutions of the recurrence equation
\[
  F_1(k_1, k_2+1)F_2(k_1, k_2) = F_1(k_1, k_2)F_2(k_1+1, k_2)
\]
has been described by Ore~\cite{Ore1930}. Note that Equation~\eqref{EQ:ss} is of this form.
The multivariate extension of Ore's theorem
was obtained by Sato~\cite{Sato1990} in the 1960s. The proofs of the piecewise and conjugate interpretations
of the discrete case of
Wilf and Zeilberger's conjecture were based on the Ore--Sato theorem~\cite{Payne1997, AbramovPetkovsek2002a}.
In his thesis~\cite{ChenThesis}, the first-named author extended the Ore--Sato theorem to the multivariate
continuous-discrete case. More recently, this result has been extended further to the case in which
also $q$-shifted variables appear~\cite{Chen2011}. To present this extension, let us recall some notation and
terminologies from~\cite{Payne1997}. For~$s, t\in \bZ$ and a sequence of expressions~{$\alpha_i$}
with~$i\in \bZ$, define
\begin{equation}\label{EQ:prod}
  \sideset{}{_s^t}\prod_i \alpha_i=\left\{
  \begin{array}{ll}
    \prod_{i=s}^{t-1} \alpha_i, & \text{if } t\geq s; \\[0.5em]
    \prod_{i=t}^{s-1} \alpha_i^{-1}, & \text{if } t<s.
  \end{array}
  \right.
\end{equation}
{We recall the Ore--Sato theorem, following the presentation of Payne's dissertation
\cite[Thm.~2.8.4]{Payne1997}.
For the proof of this theorem, one can also see \cite[pp.~6--33]{Sato1990},
\cite[Thm.~10]{AbramovPetkovsek2002a} or \cite[pp.~9--13]{Gelfand}.

\begin{theorem}[Ore--Sato theorem]\label{TH:mhg}
Let~$b_1, \ldots, b_n\in \bF(\vk)$ be nonzero compatible rational functions, i.e.,
\[
  {b_i}{S_i(b_j)}={b_j}{S_j(b_i)}, \quad \text{for~$1\leq i<j\leq n$}.
\]
Then there exist a rational function~$f\in \bF(\vk)$, constants $\mu_1,\dots,\mu_n\in \bF$,
a finite set $V\subset \bZ^n$, and for each $\vv=(v_1, \ldots, v_n) \in V$
a univariate monic rational function~$r_{\vv}\in \bF(z)$ such that
\[
  b_j = \frac{{S_j}(f)}{f} \mu_j \prod_{\vv\in V} \sideset{}{_0^{v_j}}
  \prod_{\ell} r_{\vv}( \vv \cdot \vk + \ell) \quad \text{for $j=1,\ldots ,n$}.
\]
\end{theorem}
The continuous analogue of the Ore--Sato theorem was first obtained by
Christopher~\cite{Christopher1999} for bivariate compatible rational functions
and later extended by {\.Z}o{\l}{a}dek~\cite{Zoladek1998} to the multivariate
case using the Integration Theorem of \cite[p.~5]{Cerveau1982}.
We offer a more algebraic proof using only some basic
properties of multivariate rational functions.

\begin{theorem}[Multivariate Christopher's theorem]\label{TH:mhe}
Let~$a_1, \ldots, a_m \in \bF(\vt)$ be rational functions such that
\[
  D_i(a_j) = D_j(a_i), \quad \text{for~$1 \leq i < j \leq m$.}
\]
Then there exist rational functions~$g_0,\dots,g_L \in \bF(\vt)$
and constants~$\gamma_1,\dots,\gamma_L \in \bF$ such that
\[
  a_i = D_i({ g_0}) + \sum_{\ell=1}^L \gamma_{\ell} { \frac{D_i(g_{\ell})}{g_{\ell}}}
  \quad \text{for~$i=1, \ldots, m$.}
\]
\end{theorem}
\begin{proof}
We proceed by induction on~$m$. To show the base case when~$m=1$, we apply the partial fraction decomposition over the algebraically closed field~$\bF$ to~$a_1$ and get
\[a_1 = \sum_{\ell=1}^L \sum_{j=1}^J \frac{\alpha_{\ell, j}}{(t_1 - \beta_{\ell})^j} \quad \text{
where~$\alpha_{\ell, j}, \beta_{\ell} \in \bF$ with~$\beta_{\ell} \neq \beta_{\ell'}$ for~$\ell\neq \ell'$.}\]
Then the theorem holds by taking
\[ g_0 = \sum_{\ell=1}^L\sum_{j=2}^J \frac{(1-j)^{-1}\alpha_{\ell, j}}{(t_1 - \beta_{\ell})^{j-1}}, \quad \gamma_{\ell} =\alpha_{\ell, 1}, \quad \text{and} \quad g_{\ell} = t_1-\beta_{\ell}.\]

We now assume that~$m\geq 2$ and that the theorem holds for~$m-1$. Let~$\bE$ denote the field~$\bF(t_2, \ldots, t_m)$. Over the algebraic closure~$\overline{\bE}$ of~$\bE$, the base case of the theorem allows us to decompose~$a_1$ into
\begin{equation}\label{EQ:a1}
a_1 = D_1(g_0) + \sum_{\ell=1}^L \gamma_{\ell} \frac{D_1 (g_{\ell} )}{g_{\ell}}
= D_1(g_0) + \sum_{\ell=1}^L \frac{\gamma_{\ell}}{t_1-\beta_{\ell}},
\end{equation}
where~$g_0\in \overline{\bE}(t_1)$ and, for $1\leq\ell\leq L$, we have
$g_{\ell}=t_1-\beta_{\ell}$ and $\beta_{\ell},\gamma_{\ell}\in \overline{\bE}$
such that the $\beta_{\ell}$ are pairwise distinct.

First, we claim that all~$\gamma_{\ell}$'s are actually constants in~$\bF$.
For any~$u\in \overline{\bF(\vt)}$ and $1\leq i<j \leq m$ we have
the commutation formulas
\begin{align*}
 D_i(D_j(u)) &= D_j(D_i(u)),\\
 D_i\left(\frac{D_j(u)}{u}\right)  &= D_j\left(\frac{D_i(u)}{u}\right)
\end{align*}
which, together with $\gamma_{\ell}\in \overline{\bE}$, imply that for $2\leq i \leq m$
\[D_{i}(a_1)
   =  D_1 (D_{i}(g_0)) +
   D_1 \left( \sum_{\ell=1}^L \gamma_{\ell}
 \frac{D_{i}(g_{\ell})}{g_{\ell}} \right)\\ \\
   + \sum_{\ell=1}^L D_{i}(\gamma_{\ell}) \frac{D_{1}(g_{\ell})}{g_{\ell}}.
  \]
Now it follows from the compatibility condition~$D_1(a_{i})= D_{i}(a_1)$ and after remembering $g_\ell = t_\ell - \beta_\ell$  that
\begin{equation}\label{EQ:cc3}
  D_1\left( a_{i} - D_{i}({g_0}) -  \sum_{\ell=1}^L \gamma_{\ell} \frac{D_{i}(g_{\ell})}{g_{\ell}}\right) =
  {\sum_{\ell=1}^L \frac{D_{i}(\gamma_{\ell})}{t_1-\beta_{\ell}}}.
\end{equation}
We now take, for some fixed $1\leq\ell\leq L$, the residue at
$t_1=\beta_{\ell}$ on both sides of~\eqref{EQ:cc3}: the left-hand side
vanishes as it is the derivative (with respect to~$t_1$) of a rational function, and on the right-hand
side we obtain precisely $D_{i}(\gamma_{\ell})$ since
$\beta_{\ell}\neq\beta_{\ell'}$ for $\ell\neq\ell'$.  We get that
$D_{i}(\gamma_{\ell})=0$ for $2\leq i \leq m$ and $1\leq\ell\leq L$, and therefore
the~$\gamma_{\ell}$'s are constants in~$\bF$.

Next, we claim that there always exist~$\tilde{g}_0 \in \bE(t_1)$,
$\tilde{\gamma}_{\ell}\in \bF$, and~$\tilde{g}_{\ell} \in \bE[t_1]\setminus \bE$
with $\gcd(\tilde{g}_{\ell}, \tilde{g}_{\ell'})=1$ for $\ell\neq \ell'$ such that
\begin{equation}\label{EQ:cc4}
a_1 = D_1(\tilde{g}_0) + \sum_{\ell=1}^L \tilde{\gamma}_{\ell} \frac{D_1 (\tilde g_{\ell} )}{\tilde g_{\ell}}.
\end{equation}
Let~$\bK$ be a finite normal extension of~$\bE$ containing the coefficients
of both~$g_0$ and the~$g_{\ell}$'s from~\eqref{EQ:a1} and let~$G$ be the Galois group of~$\bK$ over~$\bE$.
Since~$t_1$ is transcendental over~$\bK$, we have that~$G$ is also the Galois group of~$\bK(t_1)$
over~$\bE(t_1)$. Let~$d=|G|$. Then Equation~\eqref{EQ:a1} leads to
\[
  a_1 = D_1\left(\frac{1}{d}\sum_{\si\in G} \si(g_0)\right) + \sum_{\ell=1}^L \frac{\gamma_{\ell}}{d}
  \frac{D_1(\prod_{\si\in G} \si({g}_{\ell}))}{\prod_{\si\in G} \si({g}_{\ell})}
\]
 and the claim follows by taking
\[\tilde{g}_0 =\frac{1}{d}\sum_{\si\in G} \si(g_0) \in \bE(t_1), \quad \tilde{\gamma}_\ell =
\frac{\gamma_\ell}{d} \in \bF \quad \text{and}\quad
 \tilde{g}_{\ell} = \prod_{\si\in G} \si({g}_{\ell}) \in \bE[t_1].
\]

We have already shown that $\tilde{\gamma}_{\ell} \in \bF$, and therefore
the right-hand side of Equation~\eqref{EQ:cc3} vanishes. Thus, for
$2\leq i \leq m$, we obtain
\[
  a_{i}  = D_{i}(\tilde{g}_0) + \sum_{\ell=1}^L \tilde{\gamma}_{\ell}\frac{D_{i}(\tilde g_{\ell})}{\tilde g_{\ell}} + \bar{a}_i,
  \quad \text{for some {$\bar{a}_i\in \bE$}.}
\]
The compatibility conditions~$D_{i}(a_j)=D_j(a_i)$ imply
that~$D_{i}(\bar{a}_j)=D_j(\bar{a}_i)$ for all~$i, j$ with~$2\leq i<j\leq m$.
By the induction hypothesis, for the~$m-1$ compatible rational functions~$\bar{a}_i$,
there exist $\bar{g}_0\in \bE$, nonzero elements~$\bar{\gamma}_{\ell} \in \bF$
and~$\bar{g}_{\ell}\in \bE\setminus \bF$ for { $\ell=1, \ldots, \bar{L}$}
such that
\[
  \bar{a}_i = D_i(\bar{g}_0) + \sum_{\ell=1}^{\bar{L}} \bar{\gamma}_{\ell} \frac{D_{i}(\bar{g}_{\ell})}{\bar{g}_{\ell}},
  \quad \text{for all~$i$ with~$2\leq i  \leq m$}.
\]
 Since~$\bar{g}_0$ and the~$\bar{g}_{\ell}$'s are free of~$t_1$, we get
\[
  a_i = D_i(\tilde{g}_0 + \bar{g}_0) + \sum_{\ell=1}^L {\tilde{\gamma}_{\ell}} \frac{D_i(\tilde g_\ell)}{\tilde g_\ell} +
  \sum_{\ell=1}^{\bar{L}} \bar{\gamma}_{\ell} \frac{D_{i}(\bar{g}_{\ell})}{\bar{g}_{\ell}}, \quad \text{for all~$i$ with~$1\leq i \leq m$.} \]
This completes the proof.
\end{proof}

The next theorem describes the full structure of compatible rational functions in the general continuous-discrete setting.
}

\begin{theorem}\label{THM:crf}
Assume that~$a_1, \ldots, a_m, b_1, \ldots, b_n \in \bF(\vt, \vk)$ are compatible
rational functions with respect to~$\{D_1, \ldots, D_m, S_1, \ldots, S_n\}$. Then there exist
a rational function $f\in \bF(\vt, \vk)\setminus \{0\}$,
rational functions $g_0, \ldots, g_L, h_1, \ldots, h_n \in{\bF}(\vt) \setminus \{0\}$,
univariate rational functions~$r_{\vv}\in \bF(z)$ for each~$\vv$ in a finite set~$V\subset \bZ^n$,
and constants $\gamma_1, \ldots, \gamma_L, \mu_1, \ldots, \mu_n \in {\bF}$ such that
\begin{alignat*}{3}
 a_i & = D_i(g_0) +  \frac{D_i(f)}{f} +\sum_{\ell=1}^L \gamma_{\ell} \frac{D_i(g_{\ell})}{g_{\ell}}
   + \sum_{j=1}^n k_j \frac{D_i(h_j)}{h_j},
 & \quad & 1 \leq i \leq m,\\
 b_j & = \frac{S_j(f)}{f} \mu_j h_j \prod_{\vv\in V} \sideset{}{_0^{v_j}}\prod_{\ell} r_\vv(\vv\cdot\vk + \ell),
 && 1 \leq j \leq n.
\end{alignat*}
\end{theorem}
{
\begin{proof}
By Proposition 5.1 in~\cite{Chen2011} or Theorem 4.4.6 in~\cite{ChenThesis},
there exist~$f\in \bF(\vt, \vk)$, ${ \bar{a}_1, \ldots, \bar{a}_m}, h_1, \ldots, h_n \in \bF(\vt)$,
and ${ \bar{b}_1, \ldots, \bar{b}_n} \in \bF(\vk)$ such that
\begin{equation} \label{EQ:certt}
  a_i = \frac{D_i(f)}{f} + \sum_{j=1}^n k_j \frac{D_i(h_j)}{h_j} + \bar{a}_i
\quad \mbox{for all~$i$ with~$1 \leq i \leq m$,}
\end{equation}
and
\begin{equation} \label{EQ:certx}
  b_j = \frac{S_j(f)}{f} h_j \bar{b}_j \quad \mbox{for all~$j$ with~$1 \leq j \leq n$},
\end{equation}
and where $\bar{a}_1,\dots,\bar{a}_m$ are compatible with respect to~$\{D_1, \ldots, D_m\}$,
and $\bar{b}_1,\dots,\bar{b}_n$ are compatible with respect to~$\{S_1, \ldots, S_n\}$.
Now the full structure of the~$a_i$'s and~$b_j$'s follows from applying
Theorems~\ref{TH:mhe} and~\ref{TH:mhg} respectively to the $\bar{a}_i$'s and~$\bar{b}_j$'s.
\end{proof}
}

\section{Structure of mixed hypergeometric terms} \label{SEC:shf}
In this section, we derive the structure of hypergeometric terms
from that of their associated certificates, which are compatible rational functions.
To this end, let us recall some terminologies from~\cite{AbramovPetkovsek2002a}.
\begin{define}\label{DEF:conj}
Two hypergeometric terms~$H_1, H_2$ are said to be \emph{conjugate} if they have
the same certificates.
\end{define}
Note that if $H_1\eqa H_2$ then~$H_1$ and~$H_2$ are also conjugate to each other;
however, the converse is not true.
The first reason to introduce the notion of conjugacy is that it is the main
tool to ``correct'' Conjecture~\ref{CONJ:wz} (see Theorem~\ref{THM:main}).
As it was mentioned in the introduction the hypergeometric term $|k_1 - k_2|$
is holonomic (see Definition~\ref{DEF:holonomic}), but not proper. On the other
hand, $|k_1 - k_2|$ is conjugate to the proper term~$k_1-k_2$, but they are
not equal modulo an algebraic set.

The second reason for introducing the notion of conjugacy is related to the
inversion of sequences.  Recall that a multivariate sequence in~$\bS$ is
invertible if all its entries are nonzero.  The concept of
nonvanishing rising factorials will allow us to construct invertible
hypergeometric terms which are conjugate to those given in classical
notation (rising factorials, binomial coefficients, etc.). Recall
that the classical~\emph{rising factorial}~$(\alpha)_k$ for~$\alpha \in {\bF}$
and~$k\in \bZ$ is defined by
\[
  (\alpha)_k = \left \{
  \begin{array}{ll}
    \prod_{i=0}^{k-1} (\alpha + i), & \text{if } k\geq 0; \\[0.5em]
    \prod_{i=1}^{-k} (\alpha-i)^{-1}, & \text{if } k<0 \text{ and } \alpha\neq 1, 2, \ldots, -k;\\[0.5em]
    0, & \text{otherwise.}
  \end{array}\right.
\]
As a companion notion of rising factorials, Abramov and Petkov$\check{\rm s}$ek introduced
the \emph{nonvanishing rising factorial} $(\alpha)_k^*$ for~$\alpha \in {\bF}$
and~$k\in \bZ$ as follows:
\[
  (\alpha)_k^* = \left \{
  \begin{array}{ll}
    (\alpha)_k, & \text{if } (\alpha)_k\neq 0; \\[0.5em]
    (\alpha)_{1-\alpha}(0)_{\alpha+k}, & \text{if } \alpha\in \bZ \text{ and } \alpha >0 \text{ and } \alpha+k\leq 0;\\[0.5em]
    (\alpha)_{-\alpha}(1)_{\alpha+k-1}, & \text{if } \alpha\in \bZ \text{ and } \alpha \leq 0 \text{ and } \alpha+k>0.
  \end{array}\right.
\]
It is easy to verify that~$(\alpha)_k$ and~$(\alpha)_k^*$ are conjugate, since
they have the same certificate. Indeed, they both satisfy the recurrence
\[
  (k+\alpha)f(k+1) - (k+\alpha)^2 f(k) = 0.
\]
Similarly, we can consider factorials with an integer-linear combination of
several variables in the argument: for $(\alpha)_{\vv \cdot \vk}$ with some
fixed $\vv\in \bZ^n$, a direct calculation yields the certificate
\[
  S_i\bigl((\alpha)_{\vv \cdot \vk}\bigr) \eqa
  \Biggl(\sideset{}{_0^{v_i}}\prod_{\ell} (\alpha+\vv \cdot \vk + \ell)\Biggr) \, (\alpha)_{\vv \cdot \vk},
\]
where we use the notation introduced in~\eqref{EQ:prod}.

\begin{define}[Factorial term] \label{DEF:factorial}
A hypergeometric term~$T(\vk)\in \bS$ over~$\bF(\vk)$ is called a
\emph{factorial term} if it has the form
\[
  T(\vk) = \mu_1^{k_1} \cdots \mu_n^{k_n}
  \Biggl(\prod_{i=1}^I(\alpha_i)^*_{\vv_i\cdot\vk}\Biggr)
  \Biggl(\prod_{j=1}^J(\beta_j)^*_{\vw_j\cdot\vk}\Biggr)^{\!\!-1}
\]
where~$\mu_1, \ldots, \mu_n\in \bF$, $\alpha_i,\beta_j\in\bF$
and $\vv_i,\vw_j\in\bZ^n$ for $1\leq i\leq I$ and $1\leq j\leq J$.
\end{define}

The dictionary in Table~\ref{TAB:trans} below enables us to translate the
structure of compatible rational functions in~Theorem~\ref{THM:crf}
to that of their corresponding hypergeometric terms.

\medskip

\begin{table}[ht]
\begin{center}
\renewcommand{\arraystretch}{2.5}
\begin{tabular}{|c|c|c|c|}
  \hline
  & Hypergeometric terms & $t_i$-certificate   &  $k_j$-certificate \\ \hline
  (1) & $H_1 \cdot H_2$ & $\displaystyle{\frac{D_i(H_1)}{H_1} + \frac{D_i(H_2)}{H_2}}$  & $\displaystyle{\frac{S_j(H_1)}{H_1} \cdot  \frac{S_j(H_2)}{H_2}}$ \\[1ex] \hline
  (2) & $f(\vt, \vk)\in \bF(\vt, \vk)\setminus\{0\}$ & $\displaystyle{\frac{D_i(f)}{f}}$  & $\displaystyle{\frac{S_j(f)}{f}}$  \\[1ex] \hline
  (3) & $\exp(g_0(\vt))$ & $D_i(g_0)$ & $1$ \\[1ex]  \hline
  (4) & $\displaystyle{\prod_{\ell=1}^L g_{\ell}(\vt)^{\gamma_{\ell}}}$ & $\displaystyle{\sum_{\ell=1}^L \gamma_{\ell} \frac{D_i(g_\ell)}{g_\ell}}$ & 1\\[1ex] \hline
  (5) & $\displaystyle{\prod_{j=1}^n h_j(\vt)^{k_j}}$ & $\displaystyle{\sum_{j=1}^n k_{j} \frac{D_i(h_j)}{h_j}}$ & $h_j(\vt)$\\[1ex] \hline
  (6) & $(\alpha)^*_{\vv \cdot \vk}$ & 0 & $\displaystyle{\sideset{}{_0^{v_j}}\prod_{\ell} (\alpha+\vv \cdot \vk + \ell)}$\\[1ex]  \hline
\end{tabular}
\end{center}
\caption{Dictionary between hypergeometric terms and their certificates.}\label{TAB:trans}
\end{table}

For a rational function $g \in \bF[\vt]$ and a constant $\gamma \in \bF$, by $g^\gamma$ we mean a term
with certificate $\gamma \frac{D_i(g)}{g}$ for each $1\le i\le m$, as indicated
in row~(4) of Table~\ref{TAB:trans}, where $L=1$ is taken. In other words,
$g^\gamma$ is a solution of the system $D_i f - \gamma \frac{D_i(g)}{g} f =0\ (1 \le i \le m)$
in the unknown~$f$. Likewise, $\exp(g)$ is a solution of the system $D_i f - D_i(g) f=0\ (1 \le i \le m)$.

\begin{remark} \label{remark:exponentiation}
It is important to note that $g(\vt)^\gamma$ and $\exp({g(\vt)})$ are
determined only up to a constant, because they are defined only as solutions
of differential equations without boundary conditions.  Consequently we have
$g(\vt)^{\gamma_1} g(\vt)^{\gamma_2} = c g(\vt)^{(\gamma_1+\gamma_2)}$ for
some nonzero constant $c \in \bF$, but we do not have that $c=1$ as we would
like.  Similarly, $\left(g(\vt)^{\gamma_1}\right)^{\gamma_2}=c
g(\vt)^{(\gamma_1 \gamma_2)}$ where $c$ is not necessarily~$1$, even when
$\gamma_2$ is an integer. Analogously, the power laws for $\exp(g(\vt))$ are
different from the usual ones. In our context, however, these differences turn
out to be irrelevant. This is similar to the way how exponentials and
logarithms are introduced in the context of symbolic
integration~\cite{BronsteinBook}.  A version of the Wilf-Zeilberger conjecture
without this kind of pseudo-exponentiation remains open and might be achieved
with a more careful construction of the extension field or by considering the
case where the continuous variables are complex.
\end{remark}
\begin{theorem}\label{THM:structure}
Any hypergeometric term over~$\bF(\vt, \vk)$
is conjugate to a multivariate sequence of the form
\begin{equation}\label{EQ:mf}
  F(\vt, \vk) \exp\bigl(g_0(\vt)\bigr) \,
  \Biggl(\prod_{\ell=1}^L g_{\ell}(\vt)^{\gamma_{\ell}}\Biggr)
  \Biggl(\prod_{j=1}^n h_j(\vt)^{k_j}\Biggr) \, T(\vk)
\end{equation}
where $F(\vt,\vk)\in\bS$ is the rational sequence corresponding to some rational
function~$f\in \bF(\vt, \vk)$, and where $g_0, \ldots, g_L, h_1, \ldots, h_n\in \bF(\vt)$,
$\gamma_1, \ldots, \gamma_L\in \bF$, and~$T(\vk)$ is a factorial term that is nontrivial,
i.e., that is not equal to the zero sequence modulo an algebraic set,
in symbols: $T\neqa 0$.
\end{theorem}
\begin{proof}
This follows from Theorem~\ref{THM:crf}, Corollary 4 in~\cite{AbramovPetkovsek2002a},
and the dictionary in Table~\ref{TAB:trans}.
\end{proof}

\begin{define} \label{DEF:sf}
We call the form in~\eqref{EQ:mf} a \emph{standard form} if the denominator of the rational sequence~$F$,
i.e., the denominator of the underlying rational function~$f$,
contains no factors in~$\bF[\vt]$ and no \emph{integer-linear} factors of the form~$\alpha+\vv\cdot\vk$
with~$\vv\in \bZ^n$ and~$\alpha\in \bF$.
\end{define}

\begin{remark} \label{RE:standard}
We can always turn~\eqref{EQ:mf} into a standard form
by moving all factors in~$\bF[\vt]$ from the denominator
of~$f$ into the part~$\prod_{\ell=1}^L g_{\ell}(\vt)^{\gamma_{\ell}}$,
and moving all \emph{integer-linear} factors
into the factorial term via the formula
$\alpha+\vv\cdot \vk = (\alpha)^*_{\vv\cdot \vk+1}/(\alpha)^*_{\vv\cdot \vk}$.
\end{remark}
According to the definition by Wilf and Zeilberger~\cite{Wilf1992},
Theorem~\ref{THM:structure} distinguishes an arbitrary hypergeometric
term from a proper one as follows.

\begin{define}[Properness]\label{DEF:proper}
A hypergeometric term over~$\bF(\vt, \vk)$ is said to be~\emph{proper} if it
of the form
\begin{equation}\label{EQ:proper}
  p(\vt, \vk) \exp\!\bigl(g_0(\vt)\bigr) \,
  \Biggl(\prod_{\ell=1}^L g_{\ell}(\vt)^{\gamma_{\ell}}\Biggr)
  \Biggl(\prod_{j=1}^n h_j(\vt)^{k_j}\Biggr) \, T(\vk)
\end{equation}
where~$p$ is a sequence corresponding to some polynomial~in~$\bF[\vt, \vk]$,
and where $g_0, \ldots, g_L, h_1, \ldots, h_n\in \bF(\vt)$,
$\gamma_1, \ldots, \gamma_L\in \bF$, and where $T(\vk)\neqa 0$ is a nontrivial factorial term.
\end{define}

\begin{define}[Conjugate-Properness]
A hypergeometric term over~$\bF(\vt, \vk)$ is said to be~\emph{conjugate-proper} if it is conjugate to a proper term.
\end{define}

\begin{prop}\label{PROP:proper}
Let~$H(\vt,\vk)$ be a hypergeometric term such that~$S_j(H) \eqa H$ for all~$j$ with~$1\leq j \leq n$.
Then~$H$ is conjugate-proper.
\end{prop}
\begin{proof}
Let $a_1, \ldots, a_m, b_1, \ldots, b_n \in \bF(\vt, \vk)$ be the certificates of $H(\vt,\vk)$ with respect to $t_1, \ldots, t_m, k_1, \ldots, k_n$, respectively.
If $S_j(H) \eqa H$ for all~$j$ with~$1\leq j \leq n$, then $b_j = 1$ for all~$j$ with~$1\leq j \leq n$ by~\cite[ Proposition 1]{AbramovPetkovsek2002a}.  The compatibility conditions
${D_i(b_j)}/{b_j}  = S_j(a_i) - a_i$ for all $i, j$ with $1 \le i \le m$ and $1 \le j \le n$
imply that $a_i \in \bF(\vt)$ for all $i$ with $1 \le i \le m$. By Theorem~\ref{THM:structure}, $H$ is conjugate to a hypergeometric term of the form
$F(\vt) \exp\bigl(g_0(\vt)\bigr) \,
  \prod_{\ell=1}^L g_{\ell}(\vt)^{\gamma_{\ell}}$, where  $F, g_0, g_1, \ldots, g_L\in \bF(\vt)$ and $\gamma_{\ell} \in \bF$. By setting $g_{L+1}$ as the denominator of $F(\vt)$
 and $\gamma_{L+1}$ as $-1$, we conclude that $H$ is  conjugate to a proper hypergeometric term.
\end{proof}

\section{Holonomic Functions}
In this section, we recall some results concerning holonomic functions and
D-finite functions from~\cite{Coutinho1995, Lipshitz1989}.
The \emph{Weyl algebra}~$\Wt := \bF[\vt]\langle\Dt\rangle$ is the noncommutative polynomial ring in the variables
$\vt =t_1,\dots,t_m$ and $\Dt = D_1, \ldots, D_m$, in which the following multiplication rules hold:
\begin{alignat*}{3}
  D_iD_j &= D_jD_i, &\quad& 1\leq i, j\leq m,\\
  D_i \, p &= p \, D_i + \frac{\partial p}{\partial t_i}, && 1\leq i\leq m, \, p\in\bF[\vt].
\end{alignat*}
The Weyl algebra is the ring of linear partial differential operators with
polynomial coefficients.  Analogously, we define the \emph{Ore algebra}~$\Ot$
as the ring $\bF(\vt)\langle\Dt\rangle$ of linear partial differential
operators with rational function coefficients.

\begin{define}[Holonomicity] \label{DEF:holo1}
A finitely generated left $\bW_{\!\vt}$-module is
\emph{holonomic} if it is zero, or if it has Bernstein
dimension~$m$ (see for example~\cite[Chap.~9]{Coutinho1995}).
Let $H(\vt)$ be a function in a left $\Wt$-module
of functions. We define the annihilator of~$H$ in~$\Wt$ as
\[
  \operatorname{ann}_{\Wt}(H) := \{P\in \Wt \mid P\cdot H = 0 \},
\]
which is a left ideal in~$\Wt$. Then~$H(\vt)$ is said to be
\emph{holonomic} with respect to $\Wt$ if the left~$\Wt$-module
$\Wt/\operatorname{ann}_{\Wt}(H)$ is holonomic. Differently
stated, this means that the left ideal $\operatorname{ann}_{\Wt}(H)$
has dimension~$m$.
\end{define}
By Bernstein's inequality~\cite[Thm.~1.3]{Bernstein1971}, any finitely
generated nonzero left~$\Wt$-module has dimension at least~$m$. So
holonomicity indicates the minimality of dimension for nonzero
$\Wt$-modules, and in terms of functions this means: holonomic functions are
solutions of maximally overdetermined systems of linear partial differential
equations.

\begin{define}[D-finiteness~\cite{Lipshitz1988}] \label{DEF:dfinite}
A left ideal~$\mathcal{I}$ of~$\Ot$ is said to be D-finite if
$\dim_{\bF(\vt)} (\Ot/I) < \infty$. Assume that a function~$H(\vt)$
can be viewed as an element of a left $\Ot$-module.
Then $H(\vt)$ is said to be \emph{D-finite} with respect to~$\Ot$ if the left ideal
$\operatorname{ann}_{\Ot}(H) := \{P\in \Ot \mid P\cdot H = 0 \}$ is D-finite.
Equivalently, the vector space generated by all derivatives~$D_1^{i_1}\cdots D_m^{i_m}(H)$,
$i_1,\dots,i_m\geq0$, is finite-dimensional over~$\bF(\vt)$.
\end{define}

The next theorem shows that the notions of holonomicity and D-finiteness
coincide, which follows from two deep results of
Bernstein~\cite{Bernstein1971} and~Kashiwara~\cite{Kashiwara1978}.  An
elementary proof of the direction ``$\Longrightarrow$'' has been given by
Takayama~\cite[Thm.~2.4]{Takayama1992}. The other direction follows from the
elimination property~\cite[Lemma~4.1]{Zeilberger1990}, see also
the paragraph before Proposition~\ref{PROP:elim} below.

\begin{theorem}[Bernstein--Kashiwara equivalence]\label{THM:kashi}
Let~$\mathcal{I}$ be a left ideal of~$\Ot$. Then $\mathcal{I}$ is D-finite if
and only if~$\Wt/(\mathcal{I}\cap \Wt)$ is a holonomic $\Wt$-module.
\end{theorem}

In order to define holonomicity in the case of several continuous and
discrete variables, the concept of generating functions is employed.
The reason is that Definition~\ref{DEF:holo1} cannot be literally translated
to $\bF[\vk]\langle\Sk\rangle$, the shift analog of the Weyl algebra, since
there Bernstein's inequality does not hold.
\begin{define}\label{DEF:fgf}
For~$H(\vt,\vk)\in \bS$ we call the formal power series
\[
  G(\vt, \vz) = \sum_{k_1, k_2, \ldots, k_n \geq 0} H(\vt,\vk) z_1^{k_1}\cdots z_n^{k_n}
\]
the \emph{generating function} of~$H$.
\end{define}
The definition requires evaluating~$H$ at integer points $\vk\in\bN^n$; note
that this is always possible by the construction of~$\bS$ and the way in which the
rational functions are embedded into it, see Definition~\ref{DEF:RationalSequence}.

\begin{define}\label{DEF:holonomic}
An element $H(\vt,\vk)\in\bS$ is said to be \emph{holonomic}
with respect to~$\vt$ and~$\vk$ if its generating
function $G(\vt,\vz)$ is holonomic with respect to
$\Wtz=\bF[\vt,\vz]\langle \Dt, \Dz\rangle$.
\end{define}
We recall the notion of \emph{diagonals} of formal power series,
which will be useful for proving the following results about closure properties.
For a formal power series
\[
  G(\vz)=\sum_{i_1, \ldots, i_n \geq 0} g_{i_1,\dots,i_n} z_1^{i_1}\cdots z_n^{i_n},
\]
the primitive diagonal~$I_{z_1, z_2}(G)$ is defined as
\[
  I_{z_1, z_2}(G) :=\sum_{i_1, i_3, \ldots, i_n\geq 0} g_{i_1, i_1,\dots,i_n} z_1^{i_1}z_3^{i_3}\cdots z_n^{i_n}.
\]
Similarly, one can define the other primitive diagonals~$I_{z_i,z_j}$ for~$i<j$. By a diagonal we mean
any composition of the~$I_{z_i,z_j}$. The following theorem states that D-finiteness is
closed under the diagonal operation for formal power series.
\begin{theorem}[Lipshitz~\cite{Lipshitz1988}, 1988]\label{THM:diag}
If~$G(\vz)\in \bF[[z_1, \ldots, z_n]]$ is D-finite, then any diagonal of~$G$ is D-finite.
\end{theorem}

Zeilberger~\cite[Props.~3.1 and~3.2]{Zeilberger1990} proved that
the class of holonomic functions satisfies certain closure properties,
based on certain D-module constructions. Here, we give an alternative
proof, based on Lipshitz' work on D-finite functions.
\begin{prop}\label{PROP:closure}
Let~$H_1(\vt,\vk),H_2(\vt,\vk)\in\bS$ be holonomic.
Then both $H_1 + H_2$ and $H_1H_2$ are also holonomic.
\end{prop}
\begin{proof}
Let $G_1(\vt,\vy) = \sum_{\vk\geq\vnull} H_1(\vt,\vk)\, {\vy}^{\vk}$ and
$G_2(\vt, \vz) = \sum_{\vk\geq\vnull} H_2(\vt,\vk)\, {\vz}^{\vk}$.  By
Definition~\ref{DEF:holonomic}, $G_1$ and $G_2$ are holonomic with respect to
$\bF[\vt, \vy, \vz]\langle \Dt, \Dy, \Dz\rangle$, and therefore also D-finite
since they only involve continuous variables.  The class of D-finite functions
forms an algebra over~$\bF(\vt,\vy,\vz)$~\cite[Prop.~2.3]{Lipshitz1989}, i.e.,
it is closed under addition and multiplication.  It follows that
$G_1(\vt,\vz)+G_2(\vt,\vz)$ is also D-finite, and therefore $H_1+H_2$ is
holonomic.  Similarly, $G_1(\vt,\vy)G_2(\vt,\vz)$ is D-finite; now note that
the generating function of $H_1H_2$ is equal to the diagonal of $G_1G_2$:
\begin{align*}
  \sum_{\vk \geq\vnull}H_1(\vt,\vk) \, H_2(\vt,\vk) \, {\vy}^{\vk}
  & = I_{\vy, \vz}(G_1(\vt,\vy) \, G_2(\vt,\vz)) \\
  & = I_{y_1,z_1}(\cdots I_{y_n,z_n}(G_1(\vt,\vy) \, G_2(\vt,\vz))\cdots).
\end{align*}
By Lipshitz's theorem and Definition~\ref{DEF:holonomic}, we conclude that
$H_1H_2$ is holonomic with respect to~$\vt$ and~$\vk$.
\end{proof}

In the continuous case, Zeilberger~\cite[Lemma~4.1]{Zeilberger1990} shows that
a holonomic ideal~$\mathcal{I}$ in~$\Wt$, $\vt=t_1,\dots,t_m$, possesses the
\emph{elimination property}, i.e., for any subset of $m+1$ elements among the
$2m$~generators of $\Wt$ there exists a nonzero operator in~$\mathcal{I}$ that
involves only these $m+1$ generators and is free of the remaining $m-1$
generators.  The proof is based on a simple counting argument that employs the
Bernstein dimension.  For later use, we show a similar elimination property in
the algebra $\bF[\vt,\vk]\langle\Dt,\Sk\rangle$.

\begin{prop}\label{PROP:elim}
Let~$H(\vt,\vk)\in\bS$ be holonomic with respect to~$\vt$ and~$\vk$. Then for
any~$i\in \{1, 2, \ldots, m\}$ and~$j\in \{1, 2, \ldots, n\}$, there exists a
nonzero operator $P(\vt, k_1, \ldots, k_{j-1}, k_{j+1}, \ldots, k_n, D_i,
S_j)\in \bF[\vt, \vk]\langle \Dt, \Sk \rangle$ such that~$P(H)=0$.
\end{prop}
\begin{proof}
Without loss of generality, we may assume that $i=1$ and $j=1$. Let~$G(\vt,\vz)$
be the generating function of~$H(\vt,\vk)$ and let~$\Theta_{\ell}$ denote the Euler derivation
$z_{\ell}\frac{\partial}{\partial z_{\ell}}$ for $1\leq\ell\leq n$.
By~\cite[Lemma~2.4]{Lipshitz1989}, there exists a nonzero operator
\[
  Q(\vt, z_1, D_1, \Theta_2, \ldots, \Theta_n)\in \bF[\vt, \vz]\langle \Dt, \Dz\rangle
\]
such that~$Q(G)=0$. Write
\[
  Q = \sum_{\vw\in W} q_{\vw}(\vt, D_1)z_1^{w_1}\Theta_2^{w_2}\cdots \Theta_n^{w_n},
  \quad \text{where } W\subset\bN^n \text{ and } |W|<+\infty.
\]
Set~$u_1 = \deg_{z_1}(Q) = \max\{w_1\mid (w_1,\dots,w_m) \in W\}$, and let
$W'= \{(u_1-w_1,w_2,\dots,w_m) \mid \vw  \in W\}$.
By a straightforward calculation, we have
\begin{align*}
  Q(G) & = \sum_{\vw\in W} q_{\vw}(\vt, D_1) z_1^{w_1}\Theta_2^{w_2}\cdots \Theta_n^{w_n}
         \biggl(\sum_{\vk \geq \vnull} H(\vt,\vk) \, \vz^{\vk}\biggr)\\
  & = \sum_{\vw\in W'} q_{\vw}(\vt, D_1) z_1^{u_1-w_1}\Theta_2^{w_2}\cdots \Theta_n^{w_n}
         \biggl(\sum_{\vk \geq \vnull} H(\vt,\vk) \, \vz^{\vk}\biggr)\\
       & = \sum_{\vw\in W'} \sum_{\vk\geq \vnull} q_{\vw}(\vt, D_1) \,
         H(\vt,\vk) \, k_2^{w_2}\cdots k_n^{w_n} z_1^{k_1+u_1-w_1}z_2^{k_2}\cdots z_n^{k_n}\\
       & = z_1^{u_1}\sum_{\vw\in W'} \sum_{\genfrac{}{}{0pt}{1}{k_1 \geq -w_1}{k_2,{\dots},k_n \geq 0}} \!\! q_{\vw}(\vt, D_1) \,
         H(\vt,k_1{+}w_1,k_2,{\dots},k_n) \, k_2^{w_2}{\cdots} k_n^{w_n} z_1^{k_1}{\cdots} z_n^{k_n}\\
   & =  z_1^{u_1}\sum_{\vw\in W'} \sum_{\genfrac{}{}{0pt}{1}{k_1 \geq 0}{k_2,\dots,k_n \geq 0}} q_{\vw}(\vt, D_1) \,
         (S_1^{w_1}H)(\vk) \, k_2^{w_2}\cdots k_n^{w_n} z_1^{k_1}\cdots z_n^{k_n}\\
   &\quad \quad +z_1^{u_1}\sum_{\vw\in W'} \sum_{\genfrac{}{}{0pt}{1}{-w_1\leq k_1 < 0}{k_2,\dots,k_n \geq 0}} q_{\vw}(\vt, D_1) \,
         (S_1^{w_1}H)(\vk) \, k_2^{w_2}\cdots k_n^{w_n} z_1^{k_1}\cdots z_n^{k_n}\\
   & = z_1^{u_1} \left( \sum_{\genfrac{}{}{0pt}{1}{k_1\geq 0}{k_2,\dots,k_n \geq 0}}
           (PH)(\vk) \, z_1^{k_1}\cdots z_n^{k_n} \right) + r(\vz) = 0 \\
         \end{align*}
where $P$ is the desired operator
\[
  P = \sum_{\vw\in W'} q_{\vw}(\vt, D_1) k_2^{w_2} \cdots k_n^{w_n} S_1^{w_1}
\]
and $r(\vz)$ is a polynomial in $z_1$ of degree less than~$u_1$ with
coefficients being power series in $z_2,\dots,z_n$.  Recalling that the
extreme left member $Q(G)$ of the equality above is $0$ and noting that $r$
and the sum in the extreme right member of the equality have no powers of
$z_1$ in common and hence, no monomials $z_1^{k_1}\cdots z_n^{k_n}$ in common,
coefficient comparison with respect to $z_1^{k_1}\cdots z_n^{k_n}$ reveals
that $P(H)=0$ and $r=0$.
\end{proof}

\section{Proof of the Conjecture}
In the case of several discrete variables, piecewise and conjugate
interpretations of Conjecture~\ref{CONJ:wz} were proved
by Payne~\cite{Payne1997} and by Abramov and
Petkov\v{s}ek~\cite{AbramovPetkovsek2002a}, respectively.
In the continuous case, any multivariate hypergeometric term is
D-finite, and therefore holonomic by the Bernstein--Kashiwara equivalence. By
Proposition~\ref{PROP:proper}, it is also conjugate-proper.
Thus, Wilf and Zeilberger's
conjecture holds naturally in this case. It remains to prove that the
conjecture also holds in a mixed setting with several continuous and discrete
variables; this is done in the rest of this section. We start by proving one
direction of the equivalence in Wilf and Zeilberger's conjecture, namely that
properness implies holonomicity.

\begin{prop}\label{PROP:ptoh}
Any proper hypergeometric term over~$\bF(\vt, \vk)$ is holonomic.
Any conjugate-proper hypergeometric term over~$\bF(\vt, \vk)$ is conjugate
to a holonomic one.
\end{prop}
\begin{proof}
By Definition~\ref{DEF:proper} and Proposition~\ref{PROP:closure}, it suffices
to show that all factors in the multiplicative form~\eqref{EQ:proper} are
holonomic with respect to~$\vt$ and~$\vk$. First, we see that
\[
  \sum_{k \in \bN} k(k-1)\cdots (k-i+1) z^{k-i} = D_z^i \left( \frac{1}{1-z}\right),
\]
which is a rational function in~$z$ for each fixed~$i\in \bN$, obtained by
taking the $i$th derivative on both sides of~$\sum_{k \in \bN} z^k = 1/(1-z)$.
This fact implies that the generating function of any polynomial in~$\bF[\vt, \vk]$ is
a rational function in~$\bF(\vt, \vz)$ and therefore is holonomic.

Second, any hypergeometric term~$H(\vt)$ that depends only on the
continuous variables~$\vt$ is holonomic. According to
Definition~\ref{DEF:holonomic}, its generating function is $G(\vt, \vz) =
H(\vt) \prod_{i=1}^n (1/(1-z_i))$. Clearly $G(\vt,\vz)$ satisfies a system of
first-order linear differential equations and therefore is D-finite with
respect to $\Otz=\bF(\vt,\vz)\langle\Dt,\Dz\rangle$.  By
Theorem~\ref{THM:kashi}, the generating function~$G$ is holonomic with respect
to~$\Wtz$, and thus $H$ is holonomic with respect to $\vt$ and~$\vk$.
In particular, the factor
$\exp(g_0(\vt))\prod_{\ell=1}^Lg_{\ell}(\vt)^{\gamma_{\ell}}$
in~\eqref{EQ:proper} is holonomic.

Third, a direct calculation implies that the generating function of the factor $\prod_{j=1}^n h_j(\vt)^{k_j}$
is equal to~$\prod_{j=1}^n 1/(1-h_j(\vt)z_j)$, which is holonomic by a similar reasoning.

Finally, we have to show that the factorial term $T(\vk)$ is holonomic.  For
this, we point to \cite[Def.~6 and Thm.~3]{AbramovPetkovsek2002a}:
there, a \emph{proper term} is defined as the product of a polynomial
in~$\bF[\vk]$ and a factorial term~$T(\vk)$, and subsequently, it is shown
that every proper term is holonomic.

The second assertion now follows from
the symmetry of the conjugacy relation.
\end{proof}

The following proposition characterizes those rational functions in continuous
and discrete variables which are holonomic.
\begin{prop}\label{PROP:rat}
Let~$f(\vt, \vk)\in\bF(\vt,\vk)$ be a rational function and $F(\vt,\vk)\in\bS$ be
the corresponding rational sequence. The following statements are equivalent:
\renewcommand{\theenumi}{(\roman{enumi})}
\renewcommand{\labelenumi}{\theenumi}
\begin{enumerate}
\item\label{item.holo} $F$ is conjugate to a holonomic term.
\item\label{item.proper} $F$ is conjugate-proper.
\item\label{item.den} The denominator of~$f$ splits into the form
\[
  g(\vt)\prod_{i=1}^I (\alpha_i + \vv_i \cdot \vk)
\]
where~$g\in \bF[\vt]$, and $\alpha_i\in {\bF}$ and $\vv_i\in \bZ^n$ for $1\leq i \leq I$.
\end{enumerate}
\end{prop}
\begin{proof}
We first prove that (iii) implies (ii) and that (ii) implies (i).
Assume that the denominator of~$f$ has the form prescribed in~\ref{item.den}.
According to Remark~\ref{RE:standard} there is a proper term of the
form~\eqref{EQ:proper} that is conjugate to~$F$. Hence we have
shown~\ref{item.proper}. Now by Proposition~\ref{PROP:ptoh} we get that $F$ is
conjugate to a holonomic term, which is part~\ref{item.holo} of the assertion.

It remains to show that \ref{item.holo} implies \ref{item.den}. For this
purpose, assume that~$F$ is conjugate to a holonomic term. The rest of the
proof is divided into two parts: first it is proved that the denominator
of~$f$ splits into $g(\vt)h(\vk)$ and then it is argued that $h(\vk)$ factors
into integer-linear factors.

We may assume that $f$ is not a polynomial, otherwise the statement is
trivially true.  Let $p,d,s\in\bF[\vt,\vk]$ such that $f=p/(ds)$,
$\gcd(p,ds)=1$, and $d$ is irreducible.  We will show that~$d$ is either free
of~$\vt$ or free of~$\vk$. Suppose to the contrary that $d$ depends on both
continuous and discrete variables; without loss of generality, assume that~$d$
is neither free of~$t_1$ nor of~$k_1$. Let~$\overline{\vk}$ denote $(k_2, k_3,
\ldots, k_n)$.  Performing a pseudo-division of~$p$ by~$d$ with respect
to~$k_1$, one obtains~$e\in \bF[\vt, \overline{\vk}]$ and $q, r\in \bF[\vt,
  \vk]$ with~$\deg_{k_1}(r)< \deg_{k_1}(d)$ such that~$ep=qd+r$; note that
$r\neq0$ since $p$ and $d$ are coprime.

By Proposition~\ref{PROP:ptoh}, the embeddings of the polynomials $e$ and~$s$
into~$\bS$ are holonomic. Then the rational sequence~$R$ that corresponds to
$efs$ is conjugate to some holonomic term~$H$ by the product closure property
(Proposition~\ref{PROP:closure}), since $F$ is conjugate to a holonomic term
by assumption.

Proposition~\ref{PROP:elim} states that there exists a nonzero operator~$P$
in~$\bF[\vt, \overline{\vk}]\langle D_1, S_1\rangle$ such that~$P(H)=0$.  It
follows that $P(R) \eqa 0$, by the same argument as in the proof of
\cite[Thm.~13]{AbramovPetkovsek2002a}.  Then we have that the operator~$P$
also annihilates the rational function $efs = ep/d = q+r/d$ in $\bF(\vt,\vk)$,
since a rational function that is zero on a non-algebraic set (i.e., ``almost
everywhere''), must be identically zero.  Write
\[
  P = \sum_{i\geq0} \sum_{j\geq0} c_{i,j}(\vt,\overline{\vk}) D_1^i S_1^j
\]
where only finitely many $c_{i,j}$ are nonzero.  Since~$d$ is irreducible
and not free of~$k_1$, it is easy to see that~$S_1^j(d)$ is also irreducible
for all~$j\in\bN$ and that $\gcd\bigl(S_1^i(d), S_1^j(d)\bigr)=1$ when~$i\neq
j$.  By induction on~$i$ and noting that~$d$ is not free of~$t_1$, we have
\[
  D_1^i S_1^j \left(\frac{r}{d}\right) = \frac{r_{i,j}}{\bigl(S_1^j(d)\bigr)^{i+1}}
\]
for some polynomials $r_{i,j}\in\bF[\vt,\vk]$ for which
$\gcd\bigl(r_{i,j}, S_1^j(d)\bigr)=1$ over~$\bF(\vt, \overline{\vk})$.
Now let $j'$ be such that not all $c_{i,j'}$ are zero and choose $i'$ to be
the largest integer such that $c_{i',j'}\neq0$. Then, in the expression
\[
  P(efs) = P(q) + \sum_{i\geq0} \sum_{j\geq0} \frac{c_{i,j} r_{i,j}}{\bigl(S_1^j(d)\bigr)^{i+1}}
\]
we have a pole at $S_1^{j'}(d)$ of order $i'+1$, but this pole cannot be canceled
with any other term of $P(efs)$. This contradicts the assumption that $P$
annihilates $efs$.  Thus any irreducible factor in the denominator of~$f$ is
free of~$\vt$ or free of~$\vk$. It follows that the denominator of~$f$ can be
written as $g(\vt)h(\vk)$ with $g\in\bF[\vt]$ and $h\in\bF[\vk]$.

It remains to show that~$h(\vk)$ is a product of integer-linear factors of the
form~$\alpha + \vv \cdot \vk$.  Multiplying~$f$ by~$g$ and noting that~$g$ is
holonomic, we get that $fg=p/h$ is holonomic. We remark that a holonomic term
in $\vt$ and $\vk$ is also holonomic when viewed as a term in $\vk$ alone with
parameters~$\vt$. Then, Theorem 13 in~\cite{AbramovPetkovsek2002a} or Lemma
4.1.6. in~\cite{Payne1997} implies that $h$ factors into integer-linear
factors by regarding~$p/h$ as a holonomic term of~$\vk$ alone.
\end{proof}

We are now ready to state the main result of this paper.
\begin{theorem}
\label{THM:main}
A hypergeometric term is conjugate-proper if and only if it is conjugate to a holonomic one.
\end{theorem}
\begin{proof}
In Proposition~\ref{PROP:ptoh} it was proved that any conjugate-proper
hypergeometric term is conjugate to a holonomic one. For the other direction,
recall that Theorem~\ref{THM:structure} implies that any hypergeometric term
is conjugate to a product of a rational sequence, an exponential function, a
factorial term, and several power functions. Also in
Proposition~\ref{PROP:ptoh} it was proved that all factors in the
multiplicative form~\eqref{EQ:mf} are holonomic, except possibly
the first one, the rational sequence~$F(\vt,\vk)$.
Similarly, the reciprocals of these factors are holonomic as well.
By the product closure
property given in Proposition~\ref{PROP:closure}, we are reduced to show that
any rational sequence~$F(\vt,\vk)$ that is conjugate to a holonomic one,
is conjugate-proper. The proof is concluded by invoking Proposition~\ref{PROP:rat}.
\end{proof}

\section{Conclusion}
In this paper we have formulated and proved Wilf and Zeilberger's conjecture
in the mixed continuous-discrete setting, using the notion of conjugate
hypergeometric terms.
In the discrete setting, an alternative formulation by Payne~\cite{Payne1997}
using piecewise hypergeometric terms was employed to prove the following
version of the conjecture: a hypergeometric term is holonomic if and only if
it is piecewise-proper. It would be interesting to investigate the mixed
setting in this framework and compare it with the previous one.

As there exists a notion of $q$-proper-hypergeometric
terms~\cite[p. 589]{Wilf1992}, it would be natural to formulate and prove a
more general version of Theorem~\ref{THM:main} that also includes the
$q$-case. One important ingredient, namely the structure theorem for
compatible rational functions in all three types of variables (continuous,
discrete, and $q$-discrete) is already available~\cite{Chen2011}.

\section*{Acknowledgment}

We would like to thank Garth Payne for many fruitful discussions that helped
us to acquire a better understanding of some subtleties of hypergeometric
terms and holonomic systems. He brought to our attention several important
details in the definition of generating functions and exponentiation (see e.g.
Remark~\ref{remark:exponentiation}); we also acknowledge that the proof of
Proposition~\ref{PROP:rat} is based on some of his ideas. We are grateful to
Ruyong Feng, Ziming Li, Josef Schicho, and Michael Singer for helpful
discussions on the topic. Last but not least we want to thank the three
anonymous referees who did a particular careful reviewing and helped to
improve the exposition significantly.

\bibliographystyle{elsarticle-harv}

\end{document}